\documentclass{article}
\usepackage{amsfonts, amsmath, amscd, amsthm, amssymb}

\usepackage{mathpazo}
\usepackage{euler}
\usepackage{a4wide}
\usepackage{mathtools}

\usepackage{tikz}
\usepackage{tikz-3dplot}
\usetikzlibrary{3d}

\usepackage[all]{xy}

\newtheorem{defn}{Definition}
\newtheorem{thm}{Theorem}

\newtheorem{prop}{Proposition}
\newtheorem{lem}{Lemma}

\theoremstyle{remark}

\newtheorem{example}{Example}

\newcommand{\ZZ}{\mathbb{Z}}

\newcommand{\HH}{\mathrm{H}}
\newcommand{\A}{\mathrm{A}}
\newcommand{\RR}{\mathbb{R}}
\newcommand{\PP}{\mathbb{P}}

\newcommand{\Cplx}{\mathbb{C}}
\newcommand{\OO}{\mathcal{O}}

\newcommand{\Eu}{\mathrm{Eu}}
\newcommand{\Bl}{\mathrm{Bl}}

\usepackage{todonotes}

\usepackage{hyperref}
\definecolor{PineGreen}{rgb}{0.0,0.47,0.44}
\definecolor{MidnightBlue}{rgb}{0.1,0.1,0.44}
\definecolor{magenta}{rgb}{1.0,0.0,1.0}
\hypersetup{
  colorlinks=true,        
  linkcolor=MidnightBlue,     
  citecolor=PineGreen,        
  filecolor=magenta,      
  urlcolor=MidnightBlue           
}

\title{The Chern-Mather class of the multiview variety}

\author{Corey Harris and Daniel Lowengrub}
\begin{document}
\maketitle

\begin{abstract}
The \emph{multiview variety} associated to a collection of $N$ cameras records which sequences of image points in $\PP^{2N}$ can be obtained by taking pictures of a given world point $x\in\PP^3$ with the cameras. In order to reconstruct a scene from its picture under the different cameras it is important to be able to find the critical points of the function which measures the distance between a general point $u\in\PP^{2N}$ and the multiview variety. In this paper we calculate a specific degree $3$ polynomial that computes the number of critical points as a function of $N$. In order to do this, we construct a resolution of the multiview variety, and use it to compute its Chern-Mather class.
\end{abstract}
\section{Introduction}
Suppose that a collection of cameras are used to generate images of a scene.
The problem of \emph{triangulation} is to deduce the world coordinates of an object from its position in
each of the camera images.
If we assume that the image points are given with infinite precision, then two cameras suffice to determine the world point.
However, due to the many sources of noise in real images such as pixelization and distortion, there typically will not be
an exact solution and we will instead try to find a world point whose picture is ``as close as possible''
to the image points.

More precisely, suppose the cameras are $C_1,\dots,C_N$ and the image points are $p_1,\dots,p_n \in \RR^2$.
The goal is to find a world point $q\in\RR^3$ that minimizes the least squares error
\[
\mathrm{error}(q) = \sum_{i=1}^N(C_i(q) - p_i)^2.
\]

One application is the problem of reconstructing the 3D structure of a tourist attraction based on millions of online pictures.
It is difficult to obtain the precise configuration of any single camera, so it would not make sense to use only a small subset of them and disregard the rest.
A better approach is to solve an optimization problem which incorporates as many of the cameras as possible.
This technique was used in \cite{ROME} to reconstruct the entire city of Rome from two million online images.

Since the camera function $C_i : \RR^3 \rightarrow \RR^2$ is not linear, the standard method for solving the
triangulation problem is to first find the critical points of $\mathrm{error}(q)$ (e.g, with gradient descent),
and then select the one with the smallest error. In order to gauge the difficulty of this problem, it
is important to be able to predict the number of critical points that we expect to find for a given configuration of
cameras.

The goal of this paper is to give an explicit expression for the number of critical points of $\mathrm{error}(q)$ as
a function of the number of cameras $N$.
In fact, we compute this expression for a variation of the problem in which we allow the world points to take complex
values, and we allow these points to be in the projective space $\PP_\Cplx^3$ as opposed to the affine space $\Cplx^3$.
Our main result is that the number of critical points of $\mathrm{error}(q)$ is polynomial in the number of cameras.
\begin{thm}\label{thm:main}
The number of critical points of $\mathrm{error}(q)$ on $\PP^3_\Cplx$ is equal to
\[
p(N) = 6N^3 - 15N^2 + 11N - 4
\]
where $N$ is the number of cameras.
\end{thm}
Note that our reformulation of the problem only increases the number of possible critical points.
One can solve the original problem by first finding these points, and then discarding the ones that are not in $\RR^3$.

For a similar reason, the polynomial $p(N)$ is an upper bound on the number of critical points in the
classical triangulation problem.
In \cite{SSN}, a detailed investigation of the Lagrange multiplier equations which define the critical points is used
to compute the number of such points for $N \leq 7$.
Based on these results, it was conjectured in \cite[Conjecture 3.4]{ED} that the number of points should grow as the following polynomial:
\[
q(N) = \frac{9}{2}N^3 - \frac{21}{2}N^2 + 8N - 4.
\]
We note that our upper bound $p(N)$ is fairly close.

In order to compute the number $p(N)$, we take a slightly different perspective on the function $\mathrm{error}(q)$.
By combining the cameras $C_i:\RR^3 \rightarrow \RR^2$ we obtain a rational map
\[
\phi: \RR^3 \rightarrow \RR^{2N}.
\]

After passing to the complex numbers and taking the projective closure we obtain a rational map
\[
\phi: \PP^3_{\Cplx} \rightarrow \PP^{2N}_{\Cplx}.
\]

The image of this map is a three-dimensional variety $MV_{N} \subset \PP^{2N}$ which is known as the
\emph{multiview variety}.
We can now interpret the error function $\mathrm{error}(q)$ as measuring the distance between
a point $q \in \PP^{2N}$ and $MV_N$.
With this formulation, the number of critical points is known as the \emph{Euclidean distance degree} of the variety
$MV_N$.
The notion of ED degree was introduced in \cite{ED}, and the authors remark in \cite[ex 3.3]{ED} that the triangulation
problem was their original motivation for this concept.

In particular, by using results from \cite{ED} we prove in section \ref{sec:mather-of-mv} that this number can be
computed in terms of the Chern-Mather class $c^M(MV_N)$.
In general, the Chern-Mather class only provides an upper bound on the ED degree, but in the proof of theorem \ref{thm:ed-of_mv} we show that this inequality can be promoted to an equality for reasons specific to the multiview variety.
One advantage of this approach is that it depends only on the geometric properties of $MV_N$, and not on the specific features of the defining equations. Another advantage is that it reduces most of the difficulty to local calculations on $MV_N$.

One common way of calculating the Chern-Mather class of a singular variety $X$ is to first find a resolution
\[
\tilde{X} \xrightarrow{f} X
\]
and then analyze the singularities of $f$ in order to compare the Chern class $c(\tilde{X})$ to the Chern-Mather class $c^M(X)$

In our situation, it is natural to build a resolution of $MV_N$ by resolving the rational map $\phi$.
In section \ref{sec:resolution-of-mv}, we construct such a resolution 
\[
\tilde{\phi} : \tilde{\PP^3} \rightarrow MV_N
\]
and calculate its Chow ring and Chern class.

In order to compare the Chern class of $\tilde{\PP^3}$ to the Chern-Mather class of $MV_N$, we use the theory of \emph{higher discriminants} which was introduced in \cite{MS}.
One aspect of this theory is that it specifies which parts of the singular locus of $X$
we need to understand in order to relate $c(\tilde{X})$ to $c^M(X)$.
A precise statement is given in proposition \ref{prop:pushforward-one}.

As we show in proposition \ref{prop:hd-camera}, the higher discriminants of $\tilde{\phi}$ are
surprisingly nice.
Specifically, it turns out that in order to calculate $c^M(MV_N)$, we only have to compute the Euler obstruction of a single point $x \in MV_N$.

Moreover, in section \ref{sec:euler_obs_of_x} we show that after intersecting $MV_N$ with a hyperplane at $x$, the resulting surface singularity $(S,x)$ is \emph{taut}.
In particular, the Euler obstruction $\Eu_{MV_N}(x)$ is determined by the resolution graph of $x$ in $S$.
This allows us to use the enumerative properties of $\tilde{\PP}^3$ that are worked out in section \ref{sec:resolution-of-mv} to compute $\Eu_{MV_N}(x)$.

In the final section, we put these pieces together and obtain the polynomial $p(N)$.

\section{Definitions and notation}
Let $P$ be a $3\times 4$ matrix with values in $\RR$.
We consider each row $l$ as an affine function on $\RR^3$.
Explicitly, $l$ sends a vector $v=(x,y,z)$ to the dot product of $l$ and $(x,y,z,1)$.
We denote these functions by $f$, $g$ and $h$.

The matrix $P$ defines a rational map $\phi_P:\RR^3 \dashrightarrow \RR^2$:
\[
v \mapsto (f(v)/h(v), g(v)/h(v))
\]
which corresponds to the operation of mapping the ``world coordinates'' $\RR^3$
to the ``image coordinates'' $\RR^2$.
In other words, it describes the process of taking a picture of the world with a camera whose parameters are
encoded in $P$.

It is not hard to prove that this description of a camera is equivalent to the pinhole camera model.
In particular, the camera has a position called the \emph{camera center} and is pointing in a certain direction.
The plane defined by the camera center and direction is called the \emph{camera plane}.
 It turns out that with the above notation, the camera plane is the plane defined by the ideal $(h)$,
and the camera center is the point defined by $(f,g,h)$.
For the purposes of this paper, this observation will be taken as a definition.

Now, suppose that we have a collection of cameras $P_1,\dots,P_N$.
By taking a picture of the world with each of the cameras, we obtain a rational map:
\[
\phi_{P_1} \times\dots\times \phi_{P_N}:\RR^3 \dashrightarrow \RR^2 \times\dots\times \RR^2 \cong \RR^{2N}
\]

This map clearly extends to the complex numbers, giving us a rational map from $\Cplx^3 \dashrightarrow \Cplx^{2N}$.
Furthermore, by clearing the denominators in the definition of the maps $\phi_{P_i}$ we obtain a rational map
\[
\phi: \PP_\Cplx^3 \dashrightarrow \PP_\Cplx^{2N}
\]
defined by
\begin{equation}\label{eqn:cameramap}
\phi \left( [x:y:z:w] \right) = \left( f_1 h_2 \dots h_N : g_1 h_2 \dots h_N : \dots : h_1 \dots h_{N-1} g_N : h_1 \dots h_N \right).
\end{equation}
The scheme theoretic image of this map is called the \emph{multiview variety} associated to the cameras $P_1,\dots,P_N$.

\begin{example}
Consider the following three cameras:
\[
P_1 =
\begin{pmatrix}
  1 & 0 & 0 & 1 \\
  0 & 0 & 1 & 1  \\
  0 & 1 & 0 & 0
\end{pmatrix},
\qquad
P_2 =
\begin{pmatrix}
  0 & 1 & 0 & 1 \\
  0 & 0 & 1 & 1  \\
  1 & 0 & 0 & 0
\end{pmatrix},
\qquad
P_3 =
\begin{pmatrix}
  1 & 0 & 0 & 1 \\
  0 & 1 & 0 & 1  \\
  0 & 0 & 1 & 0
\end{pmatrix}.
\]

The associated rational map is
\[
\phi([x:y:z:w]) = [(x-w)xz : (z-w)xz : (y-w)yz : (z-w)yz : (x-w)xy : (y-w)xy : xyz]).
\]
\end{example}

We say that a collection of cameras is in \emph{general position} if the hyperplanes defined by the linear functions $ \{f_1, g_1, h_1, \dots, f_N, g_N, h_N\}$
associated to the rows of the camera matrices are in general position.

\tdplotsetmaincoords{65}{110}
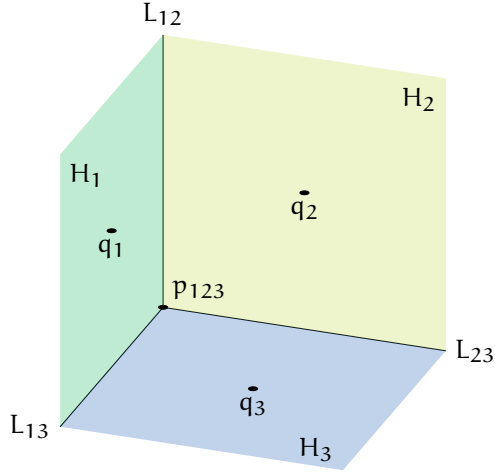
\begin{figure}[htb]
	\centering
	\begin{tikzpicture}[scale=2,tdplot_main_coords]
	    \coordinate (O) at (0,0,0);
	    \coordinate (e1) at (1,0,0);
	    \coordinate (e2) at (0,1,0);
	    \coordinate (e3) at (0,0,1);
        \draw [black] (O) -- (0,0,2);
        \draw [black] (O) -- (0,2,0);
        \draw [black] (O) -- (2,0,0);
        \path [fill opacity=0.25,fill=green!70!blue, thick] (O) -- (2,0,0) -- (2,0,2) -- (0,0,2);
        \path [fill opacity=0.25,fill=yellow!70!green, thick] (O) -- (0,2,0) -- (0,2,2) -- (0,0,2);
        \path [fill opacity=0.25,fill=blue!70!green, thick] (O) -- (0,2,0) -- (2,2,0) -- (2,0,0);
        \draw [fill] (1,0,1) circle [radius=0.03];
        \node [below] at (1,0,1) {$q_1$};
        \node [below right] at (2,0,2) {$H_1$};
        \draw [fill] (0,1,1) circle [radius=0.03];
        \node [below] at (0,1,1) {$q_2$};
        \node [below left] at (0,2,2) {$H_2$};
        \draw [fill] (1,1,0) circle [radius=0.03];
        \node [below] at (1,1,0) {$q_3$};
        \node [above left] at (2,2,0) {$H_3$};
        \draw [fill] (O) circle [radius=0.03];
        \node [above right] at (O) {$p_{123}$};
        \node [right] at (0,2,0) {$L_{23}$};
        \node [above] at (0,0,2) {$L_{12}$};
        \node [left] at (2,0,0) {$L_{13}$};
	\end{tikzpicture}
	\caption{Schematic of three cameras}
	\label{fig:cameras}
\end{figure}

Finally, we will use the following notation throughout the paper (see figure \ref{fig:cameras}).
The camera plane of the $i$-th camera will be denoted by $H_i$ and the center of the $i$-th camera
will be denoted by $q_i$.
Also, we define $L_{ij}=H_i\cap H_j$ for all $1 \leq i < j \leq N$,
and $p_{ijk} = H_i \cap H_j \cap H_k$ for all $1 \leq i < j < k \leq N$.

\section{A resolution of the multiview variety} \label{sec:resolution-of-mv}
In this section we describe a resolution of the multiview variety associated to $N$ cameras in general position.
It is obtained as an iterated blow up along smooth centers.
We then apply standard theorems to compute a presentation of the Chow ring of the resolution, and identify a couple of
important ring elements.

Let $P_1,\dots,P_N$ be camera matrices for a collection of $N$ cameras in general position, and let
\[
\phi:\PP^3 \dashrightarrow \PP^{2N}
\]
be the corresponding rational map.
We denote the associated multiview variety by $MV_N \subset \PP^{2N}$.


\begin{prop}
The base locus $B$ of $\phi$ is the reduced scheme supported on the union of the camera centers
$q_1,\dots,q_N$ and the lines $L_{ij} = H_i \cap H_j$ for all $1 \leq i < j \leq N$.
\end{prop}
\begin{proof}
It can be seen directly from the equations of $\phi$ (equation \ref{eqn:cameramap}) that $B$ is supported on the camera centers union the lines $L_{ij}$.
We will show that the scheme structure of $B$ is the reduced structure on this set.
By a strategic choice of coordinates on $\PP^3$, we can assume that $h_1=x$, $h_2=y$ and $h_3=z$.

We now analyze the scheme structure of $B$ in a neighborhood of the point $p_{123} = (x,y,z)$.
First of all, recall that the $i$-th camera contributes the two equations
$f_i \cdot \prod_{j\neq i}h_j$ and $g_i \cdot \prod_{j\neq i} h_j$
to the ideal of $B$.

By our genericity assumptions, all of the $f_i$'s, all of the $g_i$'s, and $h_i$ for $i\geq 4$ are
invertible in some Zariski neighborhood of $p_{123}$.
This implies that in a neighborhood of $p_{123}$, the ideal of $B$ has the form:
\[
(xy, xz, yz).
\]
Thus, the ideal defined by this scheme is reduced and supported on the coordinate axes.
The same argument shows that all of the lines $L_{ij}$ in the base locus have the reduced scheme structure. 
A similar argument implies the points $q_i$ are reduced.
\end{proof}

\subsection{Constructing a resolution of \texorpdfstring{$\phi$}{phi}}
In this section we construct a resolution of $MV_N$ in two stages.
First, we blow up $\PP^3$ at the points $q_1,\dots,q_N$ and at the points $p_{ijk}$ for all
$1 \leq i < j < k \leq N$. This gives us a map
\[
b_1 : Y_1 \rightarrow \PP^3.
\]
Let $\tilde{L}_{ij} \subset Y_1$ denote the proper transform of $L_{ij}$.
Note that these proper transforms are disjoint lines in $Y_1$.

For the second step, we blow up each of the lines $\tilde{L}_{ij}$ and obtain a resolution
\[
b_2: Y_2 \rightarrow Y_1
\]

Let us denote $Y_2$ by $\tilde{\PP}^3$, and denote the composition $b_1\circ b_2$ by $\pi$.
Since the pullback of the base locus $\pi^{-1}(B)$ is a Cartier divisor on $\tilde{\PP}^3$, there exists a canonical map
$\tilde{\PP}^3 \xrightarrow{\psi} \Bl_B\PP^3$ which fits into the following diagram:
\[
\xymatrix{
  \tilde{\PP}^3 \ar[r]^\psi \ar[dr]^\pi & \Bl_B\PP^3 \ar[d]^b \ar[dr]^{\Bl_B\phi} &           \\
                                      & \PP^3 \ar@{-->}[r]_\phi                &  \PP^{2N}
}
\]
were $b$ is the blowup map and $\Bl_B\phi$ is the resolution of the rational map $\phi$.

Finally, we define $\tilde{\phi} = \Bl_B\phi \circ \psi$.
Since $\tilde{\PP}^3$ is smooth, we thus obtain the following resolution of $MV_N$:
\[
\xymatrix{
  \tilde{\PP}^3 \ar[d]^\pi \ar[dr]^{\tilde{\phi}}  &    \\
  \PP^3 \ar@{-->}[r]_\phi                &  MV_N \subset \PP^{2N}
}
\]

By an abuse of notation, we will sometimes think of $\tilde{\phi}$ as a map to $\PP^{2N}$, and other times as a map to $MV_N$.

\subsection{The Chow ring of \texorpdfstring{$\tilde{\PP}^3$}{PP3~}}\label{sec:chow-ring-blowup}
Since $\tilde{\PP}^3$ is an iterated blowup of $\PP^3$ along smooth centers, we can use standard theorems to compute its Chow ring.
We will use a statement in \cite{K} which we state here for convenience.

\begin{thm}\label{thm:chow-ring-blowup}\cite[Appendix, Thm.~1]{K}
  Let $X\xrightarrow{i}Y$ be a closed embedding of smooth schemes.
  Let $\tilde{Y}$ be the blowup of $Y$ along $X$ and let $\tilde{X}$ denote the exceptional divisor.
  Suppose the map $i^* : \A^{\bullet}(Y)\rightarrow\A^{\bullet}(X)$ is surjective.
  Then, $\A^{\bullet}(\tilde{Y})$ is isomorphic to
  \[
  \frac{\A^{\bullet}(Y)[T]}{(P(T), T\cdot \mathrm{ker}(i^*))}
  \]
  where $P_{X/Y}(T) \in \A^{\bullet}(Y)[T]$ is a degree $d$ polynomial whose constant term is $[X]$, and whose restriction to $X$ is the Chern polynomial of $N_{X/Y}$.
  In other words,
  \[
  i^*P_{X/Y}(T) = T^d + c_1(N_{X/Y})T^{d-1} + \dots + c_{d-1}(N_{X/Y})T + c_d(N_{X/Y}).
  \]
  The isomorphism is induced by the map $f^*:\A^{\bullet}(Y) \rightarrow \A^{\bullet}(\tilde{Y})$,
  and by sending $-T$ to the class of the exceptional divisor.
\end{thm}

The polynomial $P_{X/Y}$ is called the \emph{Poincar\'{e} polynomial of $X$ in $Y$}.

By applying theorem \ref{thm:chow-ring-blowup} first to $Y_1 \xrightarrow{b_1} \PP^3$ and then to
$\tilde{\PP}^3=Y_2 \xrightarrow{b_2} Y_1$ we find that $\A^\bullet(\tilde{\PP}^3)$ is a quotient of the polynomial algebra
\[
A = \ZZ[\{h\} \cup \{Q_i\}_{1\leq i \leq N} \cup \{P_{ijk}\}_{1\leq i<j<k \leq N} \cup \{T_{ij}\}_{1\leq i<j \leq N}].
\]
The meaning of the generators is as follows.
Let $\tilde{q}_i \in \tilde{\PP}^3$ denote the exceptional divisor of the camera center $q_i$,
$\tilde{p}_{ijk} \in \tilde{\PP}^3$ the exceptional divisor of the point $p_{ijk}$,
and $\tilde{L}_{ij} \subset \tilde{\PP}^3$ the exceptional divisor of the line $L_{ij}$.

Then, we have the following identities in $\A^\bullet(\tilde{\PP}^3)$:
\begin{align*}
[\tilde{q}_i]     = -Q_i, && [\tilde{p}_{ijk}] = -P_{ijk}, &&
[\tilde{L}_{ij}]  = -T_{ij}.
\end{align*}

In the next section, we will need to evaluate the degree map
\[
\mathrm{deg}: \A^3(\tilde{\PP}^3) \rightarrow \ZZ.
\]
Since $\tilde{\PP}^3$ is irreducible, $\A^3(\tilde{\PP}^3)$ has rank one.
In addition, $\mathrm{deg}(h^3)=1$.
This means that calculating the degree map is equivalent to expressing every monomial $\alpha \in \A^3(\tilde{\PP}^3)$ as a multiple of $h^3$:
\[
\alpha = \mathrm{deg}(\alpha) \cdot h^3.
\]

To simplify the calculation, note that product of two generators that correspond to disjoint subschemes of $\tilde{\PP}^3$ is zero.
For example, $Q_i\cdot P_{jkl} = 0$ for all $i$, $j$, $k$ and $l$.

Thus, the main difficulty is dealing with self intersections such as $T^3_{ij}$.
In order to deal with these, we will calculate the Poincar\'{e}  polynomials of $q_i \subset Y_1$, $p_{ijk}\subset Y_1$ and $L_{ij}\subset Y_2$.
By theorem \ref{thm:chow-ring-blowup}, this will give us relations involving the self intersections, which in this case turn out to suffice for the degree calculation.

Since $q_i\subset Y_1$ is a point, its Poincar\'{e} polynomial is
\[
P_{q_i/Y_1}(Q_i) = Q_i^3 + h^3,
\]
and similarly,
\[
P_{p_{ijk}/Y_1}(P_{ijk} )= P_{ijk}^3 + h^3 .
\]

Finally, note that $L_{ij}\subset Y_1$ is a line that passes through $N-2$ blown up points.
We deduce from this that
\[
P_{L_{ij}/Y_2}(T_{ij}) = T_{ij}^2 - 2(N-3)hT_{ij} + h^2 + \sum_{k\notin\{i,j\}}P^2_{ijk}.
\]

\subsection{The Chern class of the resolution}
In this section we compute $c(\tilde{\PP}^3)$ as an element of $\A^\bullet(\tilde{\PP}^3)$ and find its pushforward to $\PP^{2N}$ (proposition \ref{prop:chern-pushforward}).
Our main tool will be the following proposition.

\begin{prop}\cite[Example 15.4.2]{F}\label{prop:fulton-chern-blowup}
  Let $Y$ be a smooth scheme and $X\subset Y$ be a closed smooth subscheme with codimension $d$.
  Consider the following blowup diagram.
  \[
  \xymatrix{
    \tilde{X} \ar[d]_g \ar[r]^j  &  \tilde{Y} \ar[d]^f \\
    X \ar[r]_i                   &  Y
  }
  \]
  Suppose that $c_k(N_{X/Y}) = i^*c_k$ for some $c_k\in \A^k(Y)$,
  and that $c(X)=i^*\alpha$ for some $\alpha\in\A^\bullet(Y)$.
  Let $\eta = c_1(\OO_{\tilde{Y}}(\tilde{X}))$.
  Then,
  \[
  c(\tilde{Y}) - f^*c(Y) = f^*(\alpha) \cdot \beta
  \]
  where
  \[
  \beta = (1 + \eta)\sum_{i=0}^d(1-\eta)^if^*c_{d-i} - \sum_{i=0}^df^*c_{d-i}.
  \]
\end{prop}

One takeaway of this proposition is that the Chern class of the blowup along a disjoint union of subvarieties is obtained by summing over contributions from the individual components.

\begin{prop}
The Chern class of the resolution $\tilde{\PP}^3$ is equal to
\[
c(\tilde{\PP}^3) = (1+h)^4 + \sum_{1 \leq i \leq N}\alpha_i + 
\sum_{1 \leq i < j \leq N}\beta_{ij} + 
\sum_{1 \leq i < j < k \leq N}\gamma_{ijk}
\]
where
\begin{align*}
\alpha_i &=  (1 - Q_i)(1 + Q_i)^3 - 1, \\
\beta_{ij} &=  (1+h)^2 \cdot [(1-T_{ij})((1+T_{ij})(-2(N-3)h) + (1+T_{ij})^2) - (1 - 2(N-3)h)], \\
\gamma_{ijk} &= (1 - P_{ijk})(1 + P_{ijk})^3 - 1.
\end{align*}
\end{prop}
\begin{proof}
Our strategy will be to use proposition \ref{prop:fulton-chern-blowup} to compute the contributions to the Chern class of each of the varieties that are blown up during the construction of $\tilde{\PP}^3$.

We first apply proposition \ref{prop:fulton-chern-blowup} to the situation where $Y=\PP^3$ and $X=q_i$ for some $i$.
In this case, we can take $c_0=1$, $c_k=0$ for $k>0$ and $\alpha=1$.
By proposition \ref{prop:fulton-chern-blowup} the blowup at $q_i$ will contribute
\[
\alpha_i = (1 - Q_i)(1 + Q_i)^3 - 1.
\]

Similarly, $\gamma_{ijk}$ represents the contribution from the blowup of the point $p_{ijk}$.

Finally, we compute the contribution from the blowup along a line $f:L_{ij}\hookrightarrow Y_1$.
Since $L_{ij}$ passes through $N-2$ of the blown up points in $Y_1$,
a quick calculation shows that we can take $c_0=1$, $c_1 = -2(N-3)h$, and the rest to be zero.
In addition, since $L_{ij}\cong\PP^1$, we can take $\alpha=(1+h)^2$.
This implies that the contribution coming from $L_{ij}$ is
\[
\beta_{ij} = (1+h)^2 \cdot [(1-T_{ij})((1-\eta)(-2(N-3)h) + (1+T_{ij})^2) - (1 - 2(N-3)h)].
\]
\end{proof}


We now compute the pullback of $c_1(\OO_{\PP^{2n}}(1))$ in $\A^\bullet(\tilde{\PP}^3)$ along the map $\tilde{\phi}$.
\begin{lem}\label{lem:pullback_c1}
The pullback of $c_1(\OO_{\PP^{2n}}(1))$ to $\tilde\PP^3$ is
\[
\tilde{\phi}^*(c_1(\OO_{\PP^{2n}}(1))) \cap [\tilde\PP^3] = N\cdot h + 2\cdot\sum_{1\leq i < j < k \leq N} P_{ijk} + \sum_{1 \leq i \leq N} Q_i + \sum_{1 \leq i < j \leq N} T_{ij}.
\]
\end{lem}
\begin{proof}
It is well known (e.g \cite[4.4]{F}) that if $L$ is a line bundle on $X$, $V\subset\HH^0(X,L)$ is a linear system and
\[
\xymatrix{
  \tilde{X} \ar[d]^\pi \ar[dr]^f &     \\
  X \ar@{-->}[r]                & \PP(V^*)
}
\]
is the induced resolution, then
\[
f^*\OO(1) = \pi^*(L) \otimes \OO(-E)
\]
where $E \subset \tilde{X}$ is the exceptional divisor.

In our case, one can show by a local calculation that the preimage of the base locus of the camera map to $\tilde{\PP}^3$
has class
\[
c_1(\OO(-E)) \cap [\tilde\PP^3] = 2\cdot\sum_{1\leq i < j < k \leq N} P_{ijk} + \sum_{1 \leq i \leq N} Q_i + \sum_{1 \leq i < j \leq N} T_{ij},
\]
so that $c_1(\tilde\phi^*(\OO(1))) = c_1(\pi^*\OO(N)) + c_1(\OO(-E))$ gives the stated expression.
\end{proof}

We can now compute the pushforward $\tilde{\phi}_*c(\tilde{\PP}^3)$ as an element of the Chow ring of $\PP^{2N}$.
\begin{prop}\label{prop:chern-pushforward}
The pushforward to $\PP^{2N}$ of $c(\tilde \PP^3)$ is 
\begin{align*}
\tilde\phi_* c(\tilde\PP^3) &=  
\left( N^3 - (4+N)\binom{N}{2} - N -2\binom{N}{3} \right) [\PP^3]
+ \left( 4N^2 - 2\binom{N}{3} - 6\binom{N}{2} - 2N \right) [\PP^2] \\
&+ \left( 6N + (N-4)\binom{N}{2} \right) [\PP^1] 
+ \left( 4 + 2N + 2\binom{N}{3} + 2\binom{N}{2} \right) [\PP^0] 
\end{align*}
\end{prop}
\begin{proof}
Since we have already calculated $c(\tilde{\PP}^3) \in \A^\bullet(\tilde{\PP}^3)$ and
$\tilde{\phi}^*(c_1(\OO_{\PP^{2n}}(1))) \in \A^\bullet(\tilde{\PP}^3)$, the calculation of $\tilde{\phi}_*c(\tilde{\PP}^3)$ is reduced to calculating
the degrees of the intersections
\[
\pi^*(c_1(\OO_{\PP^{2n}}(1)))^k \cap c(\tilde{\PP}^3) 
\]
for $0 \leq k \leq 3$.
Using the relations in $\A^\bullet(\tilde{\PP}^3)$ that we described in section \ref{sec:chow-ring-blowup}, the result follows by a direct calculation.
\end{proof}

\section{Higher discriminants}
Higher discriminants, introduced in \cite{MS}, provide a framework in which to study the singularities of a map.  In particular, we will use them to understand how the Chern class of $\tilde\PP^3$ computed above pushes forward along $\tilde\phi$.
We now recall the definitions from \cite{MS}, and phrase them in a way that will be easiest to use in our context.

\begin{defn}
  Let $f:Y \rightarrow X$ be a map of smooth manifolds.
  The \emph{$i$-th higher discriminant of the map $f$} is the locus of points $x\in X$ such that
  for every $i-1$ dimensional subspace $V\subset T_xX$, there exists a point $y\in f^{-1}(x)$ such that:
  \[
  \langle V, f_*T_yY \rangle \neq T_xX
  \]
  We denote the $i$-th higher discriminant by $\Delta^i(f)$.
\end{defn}

For example, a point $x\in X$ is in $\Delta^1(f)$ if and only if it is a critical value of $f$.
Indeed, according to the definition this happens exactly when there is a point $y\in f^{-1}(x)$ whose Jacobian
\[
J(f)_y : T_yY \rightarrow T_xX
\]
is not surjective.

On the other extreme, $x\in \Delta^{\mathrm{dim}(X)}(f)$ if and only if for every codimension one subspace $V\subset T_xX$, there exists a
point $y\in f^{-1}(x)$ that satisfies:
\[
f_*T_yY \subset V.
\]

It is instructive to consider the blow down map: $f: Y = \mathrm{Bl}_p\PP^2 \rightarrow \PP^2$.
For every point $y\in E_p = f^{-1}(p)$, $f_*T_yY$ is one dimensional. This means that $p\in \Delta^1(f)$.
In addition, it is not hard to see that for every one dimensional subspace $V\subset T_p\PP^2$, there is a point $y\in E_p$ such
that $f_*T_yY = V$.
This implies that $p\in \Delta^2(f)$.

\begin{lem}\cite[Rem.~3]{MS}
  Let $Y \rightarrow X$ be a proper map of smooth schemes. Then all of the higher discriminants of $f$ are closed,
  and we have the following stratification of $X$:
  \[
  \Delta^{\mathrm{dim}(X)}(f) \subset \dots \subset \Delta^2(f) \subset \Delta^1(f) \subset X.
  \]
  Furthermore,
  \[
  \mathrm{codim}(\Delta^i(X)) \geq i.
  \]
\end{lem}

The significance of the higher discriminants is that they tell us which strata appear when writing $f_*\mathbb{1}_Y$ in the
basis of Euler obstruction functions on $X$. For background on Euler obstructions we recommend \cite{M}.

\begin{prop}\cite[Cor.~3.3]{MS}\label{prop:pushforward-one}
  Let $f:Y\rightarrow X$ be a proper map of complex varieties. Let $\{\Delta^{i,\alpha}\}$ be the codimension $i$ components of
  $\Delta^i(f)$. Then,
  \[
  f_*\mathbb{1}_Y = \sum \eta^{i,\alpha}\Eu_{\Delta^{i,\alpha}}
  \]
  for some integers $\eta^{i,\alpha}$.
\end{prop}

\subsection{Higher discriminants of the resolution \texorpdfstring{$\tilde{\phi}$}{phi~}}
In this section we describe the higher discriminants of the map
\[
\tilde{\phi} : \tilde{\PP}^3 \rightarrow MV_N \subset \PP^{2N}.
\]
Since the definition of higher discriminants assumes that the source and target are smooth, in this section we consider $\tilde{\phi}$ as a map to $\PP^{2N}$.

Let $X_i \cong \PP^1 \subset MV_N$ denote the image of the proper transform of the camera plane of the $i$-th camera.
The restriction of $\tilde{\phi}$ to the complement of the preimage of the $X_i$'s is an isomorphism, which means that the set
theoretic singular locus of $\tilde{\phi}$ is contained in the disjoint union $\amalg_iX_i$.

The following proposition describes the higher discriminants of $\tilde{\phi}$.

\begin{prop}\label{prop:hd-camera}
  The higher discriminants of $\tilde{\phi}$ are given as follows:
  \begin{itemize}
  \item $\Delta^{2N-3}(\tilde{\phi}) = \Delta^{2N-2}(\tilde{\phi}) = MV_N$
  \item $\Delta^{2N-1}(\tilde{\phi}) = \amalg_iX_i$
  \item $\Delta^{2N}(\tilde{\phi}) = \emptyset$
  \end{itemize}
\end{prop}

To prove this proposition, we use the following lemma,
which follows almost immediately from the definition of the higher discriminants.

\begin{lem}
  Let $f: Y \rightarrow X$ be a map of smooth complex algebraic varieties.
  Let $C \subset X$ be a smooth curve.
  Suppose that the restriction of $f$ to $C$ has no critical values.
  Then $C \cap \Delta^{\mathrm{dim}(X)} = \emptyset$.
\end{lem}
\begin{proof}
Since $f|_C$ has no critical values, for every point $x\in C$ and every point $y\in f^{-1}(x)$ the one dimensional space $T_xC \subset T_xX$ is contained in $f_*T_yY$.
Therefore, if $V \subset T_xX$ is the orthogonal complement to $T_xC$, then $f_*T_yY$ is \emph{not} contained in $V$.
By definition, this implies that $x \notin \Delta^{\mathrm{dim}(X)}(X)$.
\end{proof}
We apply this lemma to each of the $\PP^1$'s $X_i \subset \PP^{2N}$.
Let $f: Y \rightarrow \PP^1 \cong X_i$ denote the restriction of $\tilde{\phi}$ to $X_i$.
Then $Y$ is isomorphic to the blowup of $\PP^2$ at $1 + \binom{N-1}{2}$ points:
$q=q_i$ and $p_{ijk}$ for $j,k \neq i$.

The map $f$ is obtained as follows.
First, let
\[
g: \mathrm{Bl}_q\PP^2 \rightarrow \PP^1
\]
be the resolution of the projection away from $q$.
Then, let
\[
h: \mathrm{Bl}_{q,p_{ijk}}(\PP^2) \rightarrow \mathrm{Bl}_p(\PP^2)
\]
be the blowup along all of the points $p_{ijk}$ for $j,k \neq i$.

Finally, we claim that $f \cong g \circ h$.
In particular, $f$ has no critical values.
According to the lemma, this proves proposition \ref{prop:hd-camera}.

\section{The Chern-Mather class of the multiview variety}\label{sec:mather-of-mv}
In this section we compute the Chern-Mather class of $MV_N$ using the theory of higher discriminants.  We then use the result to determine the ED degree of $MV_N$.

\subsection{The basic setup}
By propositions \ref{prop:pushforward-one} and \ref{prop:hd-camera}, there exists and integer $\alpha$ such that
\begin{equation}\label{eq:constructible_functions}
\tilde{\phi}_*(\mathbb{1}_{\tilde{\PP}^3}) = \Eu_{MV_N} + \alpha\cdot \sum_{i=1}^N\Eu_{X_i}.
\end{equation}
At a general point $x\in X_i$, $\chi(\tilde{\phi}^{-1}(x)) = \chi(\PP^1) = 2$ and $\Eu_{X_i}(x)=1$.
This implies that
\[
2 = \Eu_{MV_N}(x) + \alpha \Rightarrow \alpha = 2 - \Eu_{MV_N}(x).
\]

For the moment, suppose we knew the Euler obstruction $\Eu_{MV_N}(x)$.
Then, by taking the Chern-Schwartz-MacPherson  class (see \cite{M}) of both sides of equation \ref{eq:constructible_functions}
and recalling that $X_i\cong\PP^1$ we obtain
\begin{equation}\label{eqn:csm-mather}
\tilde{\phi}_*(c(\tilde{\PP}^3)) = c^M(MV_N) + (2 - \Eu_{MV_N}(x))c^M(\PP^1).
\end{equation}
Since we have already calculated $\tilde{\phi}_*(c(\tilde{\PP}^3))$ for all $N$,
this would give us the Chern-Mather class of the multiview variety $MV_N$.

\subsection{Calculating \texorpdfstring{$\Eu_{MV_N}(x)$}{Eu(x)}}\label{sec:euler_obs_of_x}
To compute $\Eu_{MV_N}(x)$, first note that we can intersect $MV_N$ with a general hypersurface $H$ passing through $x$.
As a result, we obtain a surface singularity:
\[
x \in S = MV_N \cap H.
\]
By a well known theorem about Euler obstructions (see \cite[Sec.~3]{BTS}),
\[
\Eu_{MV_N}(x) = \Eu_S(x).
\]

Now, suppose we restrict the resolution $\tilde{\phi}$ to $S$.
\begin{lem}\label{lem:ratl-norm-curve}
$\tilde{\phi}|_S$ is a resolution of $S$ such that the preimage of $x$ is a 
rational curve normal  with self intersection $-(N-1)$.
\end{lem}
\begin{proof}
Let $E$ be the preimage of $x$. Note that $E$ is the proper transform of a line 
in the camera plane of the $i$-th camera. 
To compute the self intersection of $E$ in $\tilde{S} = \tilde{\phi}^{-1}(S)$ consider the following embeddings:
\[
E \xhookrightarrow{i} \tilde{S} \xhookrightarrow{j} \tilde{\PP}^3.
\]
By the Whitney sum formula, we have
\[
(ji)_*(c(N_{E/\tilde{S}})) = (ji)_*c(N_{E/\tilde{P}^3}) \cap \tilde{\phi}^*(\OO_{\PP^{2N}}(-1)).
\]

As we have already computed $\tilde{\phi}^*(\OO_{\PP^{2N}}(1)) \in \A^{\bullet}(\tilde{\PP}^3)$, we just have to calculate
$(ji)_*c(N_{E/\tilde{P}^3})$.
By intersecting $E$ with the generators of $\A^{2}(\tilde{\PP}^3)$ we find
\[
[E] = h^2 + Q_i^2 + h\sum_{j\neq i}T_{ij}.
\]
Using this identity together with our presentation of $\A^{\bullet}(\tilde{\PP}^3)$ gives 
\[
(ji)_*c(N_{E/\tilde{P}^3}) = [E] - (N-1)h^3.
\]
Plugging everything into the Whitney sum formula shows that the degree of $c(N_{E/S})$ is $-(N-1)$, which completes the proof.
\end{proof}

We now show that this self intersection number determines the Euler obstruction $\Eu_S(x)$.
\begin{lem}\label{lem:eu-s}
With $x \in S$ the isolated singularity as above, $\Eu_S(x) = 3 - N$.
\end{lem}
\begin{proof}
Recall (\cite{L}) that a singularity germ (X,x) is \emph{taut} if the analytic type of $(X,x)$ is determined by the resolution graph of some resolution of singularities.
By \cite[2.2]{L} the vertex of the cone over the rational normal curve with degree $n$ is taut.
Let us denote this singularity by $(X_n,0)$.
Since this singularity has a resolution in which the exceptional divisor is a $\PP^1$ with self intersection $-n$, the resolution graph is a single vertex with weight $(0,-n)$. It follows that \emph{any} singularity with this resolution graph is analytically equivalent to $(X_n,0)$.

In particular, by lemma \ref{lem:ratl-norm-curve}, $(S,x)$ is analytically equivalent to $(X_{N-1},0)$
so the Euler obstruction $\Eu_S(x)$ is equal to the Euler obstruction $\Eu_{X_{N-1}}(0)$.
By \cite[3.17]{Alu}, the latter is equal to $3-N$.
\end{proof}

In conclusion, $\Eu_{MV_N}(x) = 3-N$, so equation \ref{eqn:csm-mather} becomes
\[
\tilde{\phi}_*(c(\tilde{\PP}^3)) = c^M(MV_N) + (N-1)c^M(\PP^1)
\]
By plugging in our calculation of $\tilde{\phi}_*(c(\tilde{\PP}^3))$ we 
obtain $c^M(MV_N)$.
\begin{thm}\label{thm:mather-class}
The Chern-Mather class of the multiview variety of $N$ cameras in general position is \linebreak $\sum_{i=0}^3 c_i^M(MV_N)$ where 
\begin{itemize}
\item $c^M_0(MV_N) = 4 + 4N - 2N^2 + 2\binom{N}{3} + 2\binom{N}{2}$
\item $c^M_1(MV_N) = 7N - N^2 + (N-4)\binom{N}{2}$
\item $c^M_2(MV_N) = 4N^2 - 2\binom{N}{3} - 6\binom{N}{2} - 2N$
\item $c^M_3(MV_N) = N^3 - (4+N)\binom{N}{2} - N -2\binom{N}{3}$
\end{itemize}
and $c_i^M(MV_N) = \int c^M(MV_N) \cap [\PP^{2N-i}]$.
\end{thm}

\subsection{The ED degree of the multiview variety}
As a corollary of theorem \ref{thm:mather-class}, we can compute the Euclidean distance degree of $MV_N$.
\begin{thm}\label{thm:ed-of_mv}
  The ED degree of the multiview variety of $N$ cameras in general position is equal to
  \[
  ED(MV_N) = 6N^3 - 15N^2 + 11N - 4.
  \]
\end{thm}
\begin{proof}
We can use the formula in \cite{Alu} to express the sum of the polar degrees of $MV_N$ in terms of the Chern-Mather classes.
Using this formula gives:
\[
\sum \delta_i(MV_N) = 6N^3 - 15N^2 + 11N - 4.
\]

Now, by the proof of \cite[6.11]{ED}, if $X$ is an affine cone, then the ED degree of $\overline{X}_v$
is equal to the sum of the polar classes of $\overline{X}_v$ for a general translate $X_v$ of $X$.

Suppose $MV_N$ is the multiview variety associated to the camera matrices $P_1,\dots,P_N$.
Recall that $MV_N \subset \PP^{2N}$ is the projective closure of a subvariety of $\Cplx^{2N}$  which we will call $X$.
Let $(v_1,v_2,\dots,v_{2N-1},v_{2N}) \in \Cplx^{2N}$ be a vector.
We will now show that $\overline{X}_v$ is multiview variety associated to a different collection of cameras.
Indeed, let $M_i$ be the matrix
\[
M_i =
\begin{pmatrix}
  1 & 0 & v_{2i-1} \\
  0 & 1 & v_{2i}  \\
  0 & 0 & 1
\end{pmatrix}
\]
for $1 \leq i \leq N$.
Then, the variety $\overline{X}_v$ is the multiview variety associated to the cameras
$M_i\cdot P_i$ for $1 \leq i \leq N$.

In conclusion, there exists a general configuration of cameras such that the ED degree of the
associated multiview variety $MV_N$ is equal to the sum of the polar classes of $MV_N$.
\end{proof}

\bibliographystyle{plain}
\bibliography{refs}

\end{document}